\documentclass[12pt]{amsart}
\usepackage{amsmath, amscd,amssymb}
\usepackage{graphicx,psfrag,epsfig}		
\usepackage[all]{xy}

\usepackage{url} 

\newtheorem{thm}{Theorem}[section]
\newtheorem{lem}[thm]{Lemma}
\newtheorem{prop}[thm]{Proposition}

\newtheorem{defn}[thm]{Definition}

\newtheorem{cor}[thm]{Corollary}
\newtheorem{rem}[thm]{Remark}

\newenvironment{pf}{\noindent \bf Proof:\;\rm}{\begin{flushright}$\Box$\end{flushright}}

\numberwithin{equation}{section}

\def\N{\mathbb{N}}
\def\C{\mathbb{C}}
\def\R{\mathbb{R}}

\def\K{\mathbb{K}}

\def\dis{\displaystyle}

\def\inf{\mathop{\rm inf}}
\def\det{\mathop{\rm det}}

\def\cS{{\mathcal S}}

\def\dis{\displaystyle}

\def\inf{\mathop{\rm inf}}
\def\det{\mathop{\rm det}}

\newcommand{\be}{\begin{equation}}
\newcommand{\ee}{\end{equation}}

\def \ra {\rightarrow}
\def \raa {\longrightarrow}

\def\dis{\displaystyle}

\def\bdef{\begin{Def}}
\def\endef{\end{Def}}
\def\bthm{\begin{Thm}}
\def\ethm{\end{Thm}} 
\def\blm{\begin{Lm}}
\def\elm{\end{Lm}}
\def\brm{\begin{Rem}}
\def\erm{\end{Rem}}
\def\beq{\begin{eqnarray}}
\def\eneq{\end{eqnarray}}

\def\Bbb{\mathbb}

\def\Bbb{\mathbb}
\def\R{\Bbb R}

\def\C{\Bbb C}

\begin{document}
\title{Uniform radius and regular stratifications}
\author{K. Bekka \\{IRMAR,
Universit\'e de Rennes1, 35042 Rennes (France)}}

\maketitle
\markboth{Uniform radius }{K.Bekka}

\section{Introduction}

In this paper we investigate how germs of real functions can change 
under deformation.
In particular we look at deformations of germs of isolated 
singularities from $\mathbb{R}^n$ to $\mathbb{R}^k$ ($n \geq k$)  and the relation 
with there natural stratification in some tame categorie (algebraic, 
analytic, semi-algebraic,subanalytic, $o$-minimal straucture polynomially bounded). 
The word tame in this paper will refer to one of these categories.

We say that two germs are topologically equivalent 
$f:(\mathbb{R}^n,0)\to (\mathbb{R}^k,0)$  and $g:(\mathbb{R}^n,0)\to 
(\mathbb{R}^k,0)$  are topologically 
equivalents if  there exists a germ of  homeomorphism 
$h:(\mathbb{R}^n,0)\to \mathbb{R}^n,0)$ such that $g\circ h=f$
the topological type of a germ is its right equivalence class.

A family of germs is the germ at $\{0\}\times \mathbb{R}^p$ of some 
 function $F:(\mathbb{R}^n\times \mathbb{R}^p, 
\mathbb{R}^p\times\{0\})\to (\mathbb{R}^k,0).$

We shall usually denote a family of germs by 
$f_{t}:(\mathbb{R}^n,0)\to (\mathbb{R}^k,0)$, $t\in \mathbb{R}^p$,  
where $f_{t}(x)=F(x,t).$

A stratification of a set, e.g. an algebraic variety, is roughly 
 speaking, a partition of it into smooth manifolds so that these manifolds
fit together with respect to some regularity condition. 

We are interested in regularity condition that insure topological 
triviality, which in consequence implies the constancy of the 
topological type.e  

Stratification theory was introduced
by R.Thom and H. Whitney for algebraic and analytic sets.

(see  {\cite{GM}} and {\cite{PW}} for some examples of 
applications of stratification theory). 

We consider in the paper the classical Whitney's conditions:

\begin{defn}
   \label{regular} 
Let $X,Y$ be disjoint manifolds in $\R^m$, and let $y \in Y \cap \overline{X}$. 
A triple $(X,Y,y)$ is called $(a)$ (resp. $(b)$)- regular if 

(a) when a sequence $\{x_n\} \subset X$ tends to $y$
and $T_{x_n} X$ tends in the Grassmanian bundle to 
a subspace $\tau_y$ of $\R^m$, then 
$T_y Y\subset \tau_y$;

(b) when sequences $\{x_n\} \subset X$ and 
$\{y_n\} \subset Y$ each tends to $y$, the unit vector 
$\frac{(y_n-x_n)}{|y_n-x_n|} $ tends to a vector $v$, and 
$T_{x_n}X$ tends to $\tau_y$, then 
$v\in \tau_y$.

$X$ is called $(a)$ (resp. $(b)$)- regular over 
$Y$ if each triple $(X,Y,y)$ is $(a)$(resp. $(b)$)- regular.

\end{defn}

\begin{defn}
A  stratification of a subset $V$  in $\R^m$ 
is a disjoint decomposition 
$V=\bigsqcup_{i \in I} V_i, \quad V_i\cap V_j =\emptyset \ \ \ 
\textup{for} \ \ \ i\neq j  
$
into smooth submanifolds $\{V_i\}_{i \in I}$,
called strata, is called an $(a) $(resp. $(b)$)-regular stratification  
if

1. each point has a neighborhood intersecting only finitely 
many strata;

2. the frontier $\overline{V_j}\setminus V_j$ of each stratum $V_j$ 
is a union of other strata $\bigsqcup_{i \in J(i)} V_i$;

3. any triple $(V_j,V_i,x)$ such that $x \in V_i \subset
\overline{V_j}$ is $(a)$(resp. $(b)$)-regular.

\end{defn}

\begin{thm}
    For any tame set 
$V$ in $\R^m$ there is an $a$(resp. $b$)-regular
stratification.

\end{thm}

For the proof of this theorem i.e. the existence of stratifications, 
in various tame categories, see
{\cite{Wh}}, {\cite{Th}}, {\cite{Lo}},{\cite{Hi}}, {\cite{BCR}, 
{\cite {DM}}.

For a family of germs of isolated singularities 
$F:(\mathbb{R}^n\times \mathbb{R}^p, 
\{0\}\times \mathbb{R}^p)\to (\mathbb{R},0)$, we associate the 
canonical stratification 
 of $\mathbb{R}^n\times \mathbb{R}^p$ given by the partition \\
 $$\{\mathbb{R}^n\times \mathbb{R}^p-F^{-1}(0), F^{-1}(0)-\{0\}\times \mathbb{R}^p, 
 \{0\}\times \mathbb{R}^p\}$$

We shall denote by $\pi$ the projection on the second factor, 
$V=F^{-1}(0)$, $Y=\{0\}\times \mathbb{R}^p$ , $X=V-Y$ and 
$X_{t}=\{x\in \R^n| F(x,t)=0\}.$

Since   $X_{t}$ has an isolated singularity at $(0,t)$ i.e. the critical set of the restriction of $\pi$ to $V$ is $Y.$

Then $X$ is a smooth manifold , and for each point $(x,t)\in X,$ we 
have\\
$\dis T_{(x,t)}X=\{(u,v)\in \C^n\times\C|
\sum_{i=1}^{n}u_{i}\frac{\partial F}{\partial x_{i}}(x,t) + 
\sum_{j=1}^{p}v_{j}\frac{\partial F}{\partial 
t_{j}}(x,t)=0\}=\left(\mathbb{R}d F\right)^{\perp}.$

we  use the following notation $d F=(\frac{\partial F}{\partial x_{1}},\ldots, \frac{\partial F}{\partial x_{n}},\frac{\partial F}{\partial t})$, 
$d_{x} F=(\frac{\partial F}{\partial x_{1}},\ldots, \frac{\partial F}{\partial x_{n}})$

and $\| d_{x} F\|^2=  \sum_{i=1}^{n}\|\frac{\partial F}{\partial x_{i}} \|^2.$

For the canonical stratification associated to  a family of germs of 
isolated singularities, to be (a) ( resp. (b)) regular,  can be made
more practical by the following form:

\begin{defn}
  We say that $F$ is Whitney regular at $0 $ if its canonical 
  stratification is Whitney regular and this is equivalent to the following conditions are satisfied:
 
\underline {  condition $(a)$ :}

$$\lim_{(x,t)\to 0\atop (x,t)\in X}\left ( \frac{\frac{\partial 
F}{\partial t_{j}}(x,t)}{\| d_{x} F(x,t)\|}\right)=0 \qquad \text{for 
each } 1\leq j\leq p.$$

\underline { condition $(b^\prime)$: }

$$\lim_{(x,t)\to 0\atop (x,t)\in X}\left (\frac{\sum_{i=1}^{n}x_{i}\frac{\partial F}{\partial x_{i}}(x,t)}{\|x\|\| d_{x} F(x,t)\|}\right)=0.$$

\end{defn}

\begin{rem}
    \begin{enumerate}
        \item It is known that condition $a + b^\prime$ is equivalent to 
    condition $b$. (see \cite{Th}, \cite{Ma})
    
        \item A Whitney regular family of function germs is 
	topologically trivial, then the topological type is constant 
	in such family. (see \cite{Th}, \cite{Ma})

    \end{enumerate}
    
\end{rem}

\section{Uniform radius and vanishing folds}

A family of germs is said to have no coalessing of critical points  (in the sens of 
H.King \cite{K}) if there exists a neighbourhood $U$ of 
$\{0\}\times\mathbb{R}^p$ in $\mathbb{R}^n\times \mathbb{R}^p$, such 
that the restriction of $f_{t}$ to $U\cap 
(\mathbb{R}^n-\{0\}\times\{t\} )$ has no critical point (i.e. submersion) for 
each $t\in \mathbb{R}^p.$

For example the family $f_{t}(x)=x^3-3tx^2$ has a coalesing, and 
$f_{0}$ is topologically equivalent to the identity, but for
$t\not=0$, $f_{t}$ has a maximum or  a minimum at $0$ and then fails 
to be topologically equivalent to the identity.

For a germ of isolated singularity $f:(\mathbb{R}^n,0)\to (\mathbb{R},0)$ we denote by  
$\mu(f_t)=\dim\frac{\mathcal{O}_n}{(\frac{\partial f}{\partial 
x_1},\ldots, \frac{\partial f}{\partial x_n})}$ its the Milnor number.

We say that a family of isolated singularities is $\mu$-constant 
family if $\mu(f_{t})=\mu(f_{0})$ for each $t\in \mathbb{R}^p.$

\begin{rem}
      For complex analytic  germs no coalescing of critical 
      points is equivalent to  the family is $\mu$-constant.
\end{rem}

\begin{defn}
   Let $f:(\mathbb{R}^n,0)\to (\mathbb{R},0)$ be a tame germ with isolated 
singularity. Let $\rho:(\mathbb{R}^n,0)\to \mathbb{R}^+$ be a germ of 
a  tame submersion such that $\rho^{-1}(0)=\{0\}$.

The $\rho$-Milnor radius of $f$,$\rho(f)$, is the smallest critical value of 
the restriction of $\rho$ to the smooth variety $f^{-1}(0)-\{0\}$  
i.e.\\  $\rho(f)= \inf \{\rho(x)| x\in  f^{-1}(0)-\{0\}  \text{ and } 
d_{x} f=\lambda d_{x} \rho, \text{ for some } \lambda\in 
\mathbb{R}\}$.

If there are no critical values then $\rho(f)=\infty$.

\end{defn}

To extend this notion to tame maps $f:(\mathbb{R}^n,0)\to 
(\mathbb{R}^k,0)$  we use the following notation.

\begin{defn}
   Let $f:(\mathbb{R}^n,0)\to (\mathbb{R}^k,0)$ be a tame germ with isolated 
singularity. Let $\rho:(\mathbb{R}^n,0)\to \mathbb{R}^+$ be a germ of 
a  tame submersion such that $\rho^{-1}(0)=\{0\}$.

The $\rho$-Milnor radius of $f$, $\rho(f)$, is the smallest critical value of 
the restriction of $\rho$ to the smooth variety $f^{-1}(0)-\{0\}$  
i.e.\\ $\dis \rho(f)= \inf \{\rho(x)| x\in  f^{-1}(0)-\{0\}  \text{ 
and } \sum_{1\leq i_1<\ldots<i_{m+1}\leq 
n} \left|\frac{D(f_1,\ldots, 
f_k,\rho)}{D(x_{i_1},\ldots,x_{i_{m+1}})}(x)\right|^2 =0\}.$

If there are no critical values then $\rho(f)=\infty$.

\end{defn}

\begin{defn}[uniform Milnor radius]
    
    Let $\{f_{t}\}$, $t\in \mathbb{R}^p$ be a  family of germs at 
    $\{0\}$ of isolated singularity.

    We say that $\{f_{t}\}$ has uniform $\rho$-Milnor radius if there 
    is an $\epsilon>0$ such that $\rho(f_{t})> \epsilon$ for all $t\in 
    \mathbb{R}^p.$
   
\end{defn}

We call such function $\rho$ a control function.
\vskip 1cm

We say that a point $p \in f^{-1}(0) $ is a $\rho$-Kink ( or simply a { Kink})  of $ f^{-1}(0) $
if $p$  is non singular point of  $f$  and if $p$ is a critical point of $\rho$ restricted to the manifold of smooth points of $ f^{-1}(0).$
(see \cite{O})

\vskip 0.5cm
\begin{rem} For $k=1$, 

an easy computation shows that a nonsingular $p\in f^{-1}(0)$  is a 
kink if and only if $df(p)=\lambda d_x\rho(p)$ for some $\lambda$ in $\mathbb{R}-\{0\} .$
\end{rem}

\vskip 0.5cm

We suppose that for every $t\in\mathbb{R}^p$, $f_t(0)=0$ and $0$ is an isolated critical point of $f_t.$

Let $\gamma:[0,\epsilon] \to\R^n\times[0,1]$ be a real analytic path $\gamma(s)=(x(s),t(s))$  such that:
\begin{itemize}
\item[1)] $\gamma(0)=(0,0)$
\item[2)] $\vert x(s)\vert>0$ and $\vert t(s)\vert>0$ for all $0<s<\epsilon,$ and
\item[3)] $f(x(s),t(s))=0$ for all $0\leq s\leq \epsilon.$
\end{itemize}

\vskip 5pt

\begin{defn}

The path $\gamma$ will be called a { $\rho$-vanishing fold}  of $f$  
(centered at $0$ ) if  $x(s)$ is a $\rho$-kink of $f_{t(s)}^{-1}(0)$
 for every  $s\in (0,\epsilon].$
\end{defn}

\begin{prop}
 $\{f_{t}\} $ has a $\rho$-uniform radius if and only if it has no 
 $\rho$-vanishing fold in 
 $U(\epsilon_0)=\{x\in\K^n: \rho(x)<\epsilon_0\} $  for some $\epsilon_0 >0.$
\end{prop}

\begin{rem}
    In the analytic complex case we have the following theorem 
    which relies the jump of Milnor number with the existence of 
    vanishing folds.

  We obtain a generalisation of \cite{O}: 
    
\begin{thm} $\K=\C.$ 

Let $F:(\mathbb{C}^n\times \mathbb{C}, 
\{0\}\times \mathbb{C})\to (\mathbb{C},0)$ a family of germs of 
isolated singularities  and  $X_{t}=\{f_t^{-1}(0)\}$ the 
corresponding family of hypersurfaces.
Let $\mu_t\equiv \mu $ be the Milnor number of $f_t$ at the origine 
and suppose that $\mu_t=\mu$ is constant for $0<t\leq 1$ and $\mu<\mu_0$ .

Then, the family $\{f_t^{-1}(0)\}$ admits a vanishing fold centered at $0$.
\end{thm}

\end{rem}

\begin{rem}
It is not difficult, to see that any family of isolated singularity 
of quasihomogeneous  functions can not have a vanishing fold with 
$\rho=\sum_{i=1}^n |x_i|^2.$ 
i.e. has $\rho$-Milnor uniforme radius.

In fact,  if $\{f_t\}$ is quasihomogeneous  family of type
$(w_1,\ldots,w_n; D)$, the "Euler formula" gives:

$$\sum_{i=1}^n w_i x_i \frac{\partial f_t}{\partial x_i}=D. f_t.$$

And now if $f$ has a kink at $p=(x,t)$  then there exists $\lambda\in 
\R^*$ such that

$d_xf(p)=\lambda {d_x\rho(p)}$ i.e. $\frac{\partial f_t}{\partial x_i}=\lambda {x_i}$

Since $p\in f^{-1 }(0)$ the Euler formula gives
$$\lambda\sum_{i=1}^n w_i \vert x_i\vert^2 =0$$ which implies 
$p\in\{0\}\times \mathbb{R}^p$
then it's not in the smooth part.

We can also show that it has uniform $\rho$-milnor radius, for  
$\rho$ a quasihomogeneous control function with respect to the system 
of weights,  for example: $\sum_{i=1}^n  \vert x_i\vert ^{2\frac{D}{w_i}}.$

\end{rem}

\begin{rem}
Having  a uniform $\rho$-Milnor radius for some ``control function'' 
$\rho$ doesn't  for a family of germs is  in general weaker than whitney regular.

For example,  take the Brian\c{c}on-speder family  $F(x,y,z,t)= 
z^5+ty^6z+y^7x+x^{15}$, its a family
quasihomogeneous polynomials of type $(1,2,3; 15)$ 

The Milnor number is  given   for an isolated  singularity
quasihomogeneous  $f$ of type
$(w_1,\ldots,w_n; D)$ by the formula:

$$\mu(f)=\frac{(D-w_1)(D-w_2)\ldots(D-w_n)}{w_1.w_2.\ldots w_n}$$

which gives in our case $\mu(f_t)=364.$

The family $f_t$ is $\mu$-constant but for a generic hyperplan $H$ in 
$\mathbb{R}^3$, its equation can be written
 $z=ax+by$ with $a,b\in\R-\{0\}$, and so the restriction family 
 $g_t=f_t\vert _{H} $ is the family of polynomials
 $g_t(x,y)=x^5+txy^6+y^7(ax+by)+(ax+by)^{15}$

 Then, for $t\not=0,$  $g_t$ is semiquasihomogeneous with leading 
 term $x^5+txy^6$ is of type $(3,2:15).$
 Using the fact that $\mu(g_t)=\mu(x^5+txy^6)$ we obtains
 $$\mu(H\cap Y_t)=26.$$
 But for $t=0$, $g_0(x,y)=x^5+y^7(ax+by)+(ax+by)^{15}=x^5+by^8 + (axy^7+(ax+by)^{15})$
 is semiquasihomogeneous with leading term $x^5+by^8$ is of type $(8,5:40)$
then:  $$\mu(H\cap Y_0)=28.$$

Since the Milnor number jumps, this family has $\rho$-vanishing fold.

In fact the family of curves we obtain by the intersection with a generic hyperplan must have a vanishing folds.

In the complex, Brian\c{c}on and Speder show that this family is not 
Whitney regular ( using the fact that whitney regular family must have 
constant milnor numbers after intersection by generic hyperplan ).
\end{rem}

\section{Vanishing folds and Whitney condition}

Let $F$ an analytic function from $\R^n\times \R$ to $\R$, in an neighbourhood of $0$

$$\begin{matrix} F: & \R^n\times \R^p,0 &\to& \R,0 \\ & (x,t)&\mapsto & F(x,t)  \end{matrix}$$
$F(0,t)=0$

We denote by $\pi$ the projection on the second factor, 
$V=F^{-1}(0)$, $Y=\{0\}\times \R$ and $X_{t}=\{x\in \R^n/ F(x,t)=0\}.$

We suppose  $X_{t}$ has an isolated singularity at $(0,t)$ i.e. the critical set of the restriction of $\pi$ to $V$ is $Y.$

Then $X=V-Y$ is an analytic complex manifold of dimension $n$, and for each point $(x,t)\in X$ we have

$\dis T_{(x,t)}X=\{(u,v)\in \R^n\times\R^p/ 
\sum_{i=1}^{n}u_{i}\frac{\partial F}{\partial x_{i}}(x,t) 
+\sum_{j=1}^{p}u_{i}v_{j}\frac{\partial F}{\partial t}(x,t)=0\}=\left(\R {d} F\right)^{\perp}.$
 
 where $d F=(\frac{\partial F}{\partial x_{1}},\ldots, \frac{\partial 
 F}{\partial x_{n}},\frac{\partial F}{\partial t})$, $d_{x} F=(\frac{\partial F}{\partial x_{1}},\ldots, \frac{\partial F}{\partial x_{n}}).$

Let $\mathcal{G} $  be the  set of analytic applications germs from 
$\R^n\times\R^p,0$ to $\R^n\times\R^p,0$  of the following type :

$\Phi(y,\tau)=(\Psi(y,\tau),\lambda(\tau))=(x,t)$,

where $\Psi$ for small $\tau$ is a germ of automorphisms of $(\R^n,0)$
(i.e. $\det \left( \frac{\partial \Psi}{\partial y}\right)\not=0$ and $\Psi(0,\tau)=0$ ).

We suppose given  $F: \R^n\times\R^p, 0\to \R,0$ is an analytic 
deformation of $f=f_{0}$ such that
$F(0,t)=\frac{\partial F}{\partial x_1}=\ldots=\frac{\partial F}{\partial x_n}=0$, $X=F^{-1}(0)$,
$X_t=f_t^{-1}(0)$  and $Y=\{0\}\times\R$.

The following theorem says that Whitney regularity is equivalent to 
the stability of the uniform  $\rho$-Milnor property with respect to 
families of  linear change of variable in $x$.

\begin{thm}\label{mainthm}
Let $F$ be a $\mu$-constant deformation. The following conditions are equivalent
\begin{itemize}
\item[(i)] $F$ is Whitney regular
\item[(ii)] $\forall \Phi\in \mathcal{G}$, $F\circ \Phi$ has no vanishing fold.
\end{itemize}
\end{thm}

\pf
(i) $\Rightarrow$   (ii)

We have seen that, if a deformation is Whitney regular then it's has no vanishing folds.
We have then only to show that if $F$ is Whitney then so is $F\circ\Phi$ for all $\Phi\in\mathcal{G}.$

By definition $F\circ \Phi(y,\tau)=F(\Psi(y,\tau),\lambda(\tau))$; this suggest to do it in the two following steps:

Firstly, for $\lambda=Id_{\mathbb{R}^p}$, $\Phi_1(y,\tau)=(\Psi(y,\tau),\tau)$, is 
then an analytic diffeomorphism of $\mathbb{R}^{n+p}$, since Whitney's 
conditions are invariant by diffeomorphism (see \cite{Ma}), if
$F$ is Whitney regular, so is $F\circ\Phi_1$, where 
$\Phi_1(y,\tau)=(\Psi(y,\tau),\tau).$

Secondly,  if $F$ is Whitney regular, then so is $F\circ\Phi_2$, where $\Phi_2(y,\tau)=(y,\lambda(\tau))$
and $\lambda: \R^p,0\to\R^p,0.$

In fact, the condition $b^\prime$ is trivially satisfied since it does not make use of the partial derivative
relatively to the parameter.

To check, the $(a)$ condition, we compute
$\dis\frac{\partial F\circ \Phi_2}{\partial 
t_{j}}(y,\lambda(\tau))=\sum_{m=1}^p\frac{\partial\lambda_{m}}{\partial t_{j}}(\tau)
\frac{\partial F}{\partial \lambda_{m}}(y,\lambda(\tau))$,
since $F$ satisfy $(a)$ condition, we have
$$\lim_{(y,\tau)\to 0\atop (x,y)\in X-Y}\left ( \frac{\frac{\partial 
F\circ\Phi_2}{\partial t_{j}}(y,\tau)}
{\| d_{x} F\circ\Phi_2(y,\tau)\|}\right)=\lim_{(y,\tau)\to 0\atop 
(x,y)\in X-Y}\sum_{m=1}^p\frac{\partial\lambda_{m}}{\partial 
t_{j}}(\tau)
\left ( \frac{\frac{\partial F\circ\Phi_2}{\partial t_{m}}(y,\tau)}
{\| d_{x} F\circ\Phi_2(y,\tau)\|}\right)=0.$$

Now $(1)$ and $(2)$ implies that for any $\Phi\in \mathcal{G}$, 
$F\circ\Phi$ is Whitney regular, then it has no vanishing folds.

(ii) $\Rightarrow$ (i)

Firstly,  since $F$ is a $\mu$-constant deformation in a neighborhood of $0$, 
 its satisfy the  $(a)$ regularity condition (  in fact we have more, 
 $\mu$-constant implie ``good stratification'' in the sens of Thom, see \cite{LS}, 
 \cite{BS}, \cite{T}).

 Let us suppose that $b$ fails, which in turn implies  $b'$ fails, 
  since $a$ holds. 
 
Let $\Delta(z,\tau)=\frac{\sum_{i=1}^{n}x_{i}\frac{\partial F}{\partial x_{i}}(z,\tau)}{\|x\|\| grad_{x} F(z,\tau)\|}$
where $(z,\tau)\in X-Y.$

Then there exists a real analytic curve
 $\gamma:[0,\epsilon] \to X$, $\gamma(s)=(x(s),t(s))$   and $\delta_0>0$ such that:
\begin{itemize}
\item[1)] $\gamma(0)=(0,0)$
\item[2)] $f(z(s),\tau(s))=0$ for all $0\leq s\leq \epsilon,$ and
\item[3)] $\lim_{s\to 0}  \Delta(z(s),\tau(s))=l\not=0$
\end{itemize}

Let us dnote par $v$ the valuation of $\mathcal{O}_{\R,0}$ associated to $\gamma.$

We will use the following notations:\\
$\displaystyle v(z)= \inf_{1\leq i\leq n} z_i$  for $z\in \R^n,$
$\displaystyle v(\frac{\partial f}{\partial x})= \inf_{1\leq i\leq 
n}v(\frac{\partial f}{\partial x_i}).$\\

In this conditions, if we denote $v(z)=p$ and $v(\frac{\partial f}{\partial x})=q$, we can suppose
(change the order of variables if needed) that $v(z_1)=p.$

Since, $\Delta\circ \gamma(s)$ has a non zero limit when $s$ tends to $0$, we may conclude that
$v(<z,\overline{\frac{\partial f}{\partial x}}(z,\tau)>=v(z)+v(\frac{\partial f}{\partial x}(z,\tau))=p+q$

Let us now  denotes  $\gamma(s)=(p_1(s),\ldots,p_n(s), \lambda(s)),$

$\overline{\frac{\partial f}{\partial x}}\circ\gamma(s)=(q_1(s),\ldots,q_n(s) $
and
$v(<z,\overline{\frac{\partial f}{\partial 
x}}>\circ\gamma(s))=u(s).$\\
Let us now define\\
$\Phi: \R^n\times \R, 0\to \R^n\times \R, 0 $ by: $\Phi(y_1,\ldots,y_n, \tau)=(\Psi(y, \tau), \lambda(\tau))$

with $ \Psi(y, \tau)=(y_1-\frac{ p_1}{u}h,y_2+\frac{ p_2}{p_1}y_1-\frac{ p_2}{u}h,\ldots,
y_n+\frac{ p_n}{p_1}y_1-\frac{ p_n}{u}h) $

where $h=q_2y_2+q_3y_3+\ldots+q_ny_n.$

We may first check that, $\Phi\in \mathcal{G}:$

1) $\Phi$ is analytic.

We use for this the valuation along $\gamma.$

If $j\not = 1,$ then for $y_j+\frac{ p_j}{p_1}y_1-\frac{ p_j}{u}h$ , we have by hypothesis
$v(\frac{ p_j}{p_1}y_1)\geq v(p_1)+v(y_1)-v(p_1)\geq0$ and
$v(\frac{ p_j}{u}h)= v(p_j)+v(h)-v(u)\geq (p+q)-(p+q)=0.$

If $j=1 $ for $y_1-\frac{ p_1}{u}h$, we have
$v(\frac{ p_1}{u}h)=v(p_1)+v(h)-v(u)=(p+q)-(p+q)=0.$

2) $\Psi(0, \tau)=0$
 
3) The jacobian of $\Psi$ is invertible in a neighborhood of $0.$   

For this we compute the determinant of this Jacobian and show it equals $1.$

Let $\Phi_1(y,\tau)=y_1-\frac{ p_1}{u}h$ and
$\Phi_j(y,\tau)=y_j+\frac{ p_j}{p_1}y_1-\frac{ p_j}{u}h$ for $2\leq j\leq n.$

Then $\frac{\partial \Phi_1}{\partial y_1}=1$ and $\forall j\geq 2,$

$\frac{\partial \Phi_1}{\partial y_j}=
-\frac{ p_1}{u}q_j.$

 $\forall i,j\geq 2, $ $i\not= j,$   $\frac{\partial \Phi_j}{\partial y_i}=
-\frac{ p_j}{u}q_i$

 $\forall i\geq 2, $   $\frac{\partial \Phi_i}{\partial y_i}=1-\frac{ p_i}{u}q_i$
 and   $\frac{\partial \Phi_i}{\partial y_1}=\frac{ p_i}{p_1}.$

\begin{equation}
 \det\left(\frac{\partial \Psi}{\partial y}(y, \tau)\right)=\left\vert\begin{array}{cccccc}
1 & -\frac{ p_1}{u}q_2 & \ldots & -\frac{ p_1}{u}q_i  & \ldots & -\frac{ p_1}{u}q_n\\
\frac{ p_2}{p_1} & 1-\frac{ p_2}{u}q_2 & & \vdots &  &\vdots\\
 \vdots& & \ddots & &  &\vdots \\
 \vdots & & & 1-\frac{ p_i}{u}q_i &  & \vdots\\
 \vdots & & &  &  &  \vdots\\
 \vdots & &  & & \ddots &  \vdots\\
\frac{ p_n}{p_1} & -\frac{ p_n}{u}q_2& \ldots & -\frac{ p_n}{u}q_i &\ldots &1-\frac{ p_n}{u}q_n
 \end{array}\right\vert
 \end{equation}

If we denotes by $C_j$ the jth row, and applied the transformation $C_j=C_j+\frac{ p_iq_i}{u}C_1$ for $j=1\cdots
m$

we see that 
\begin{equation}
 \det\left(\frac{\partial \Psi}{\partial y}(y, \tau)\right)=\left\vert\begin{array}{cccccc}
1 & 0 & \ldots & \ldots & \ldots & 0\\
\frac{ p_2}{p_1} & 1 & & \vdots &  &\vdots\\
 \vdots&0 & \ddots &0&  &\vdots \\
 \vdots & & & 1&  & \vdots\\
 \vdots & & &  &  &  \vdots\\
 \vdots & &  & & \ddots &  0\\
\frac{ p_n}{p_1} &0& \ldots & \ldots &0 &1
 \end{array}\right \vert =1
 \end{equation}

We can now conclude that $\Phi\in \mathcal{G}.$

Moreover, by construction we have
$\Phi(p_1(\tau),0\ldots,0,\tau)=\gamma(\tau)$

The computation gives us

$\frac{\partial F\circ\Phi}{\partial y_{i}}(p_1(\tau),0\ldots,0,\tau)=
\sum_{j=1}^{n}\frac{\partial F}{\partial x_{i}}(\gamma(\tau))
\displaystyle \frac{\partial\Phi_j}{\partial y_{i}}(p_1(\tau),0\ldots,0,\tau)$

$\displaystyle =\sum_{j=1}^{n}q_j(\tau)\frac{\partial\Phi_j}{\partial y_{i}}(p_1(\tau),0\ldots,0,\tau)$

$\displaystyle = -\frac{ q_1(\tau)q_i(\tau) p_1(\tau)}{u}+ q_i(\tau)\left(1- \frac{ p_i(\tau)q_i(\tau)}{u}\right)
+\sum_{j\not=1\atop j\not= i}-\frac{ q_j(\tau)q_i(\tau) p_j(\tau)}{u} $

$\displaystyle =-\frac{ q_i(\tau)}{u}\left(p_1(\tau)q_1(\tau) +
p_i(\tau)q_i(\tau)+\sum_{j\not=1\atop j\not= i}p_j(\tau)q_j(\tau)\right)
+q_i(\tau)$

$=-\frac{ q_i(\tau)}{u}.u +q_i(\tau)=0.$

If $i=1,$  $\displaystyle I=q_1(\tau)+\sum_{j=2}^{n}q_j(\tau)\frac{ p_i(\tau)}{ p_1(\tau)} =\frac{u}{ p_1}$

Then, we obtain that
$\frac{\partial F\circ\Phi}{\partial 
y_{i}}(p_1(\tau),0\ldots,0,\tau)=\lambda({p_1(\tau)},0,\ldots,0)$
with $\lambda=\frac{u}{\vert p_1\vert^2}$
this means that $ F\circ\Phi$ has a vanishing fold.

\begin{rem}
In this proof we can replace $\mathcal{G}$ by the set

$\mathcal{G}_l=\{\Phi=(\Psi,\lambda):
\R^n\times\R,0 \to \R^n\times\R,0\; \text{such that} \; \Psi(.,\tau)\in Gl(\R^n)\} $
\end{rem}

A consequence of  this theorem is, an example of $\mu$-constant deformation with a vanishing fold gives an example of
non Whitney regular $\mu$-constant deformation.

\subsection{Example}

\underline { The Brian\c{c}on and Speder  example has vanishing 
folds  } (see \cite{BS} \cite{Tr})

From the theorem,  Whitney faults is detected by vanishing folds.
So to find a vanishing fold, it suffises to find an arc along which the Whitney regularity fails.

Let  $F(x,y,z,t)= z^5+ty^6z+y^7x+x^{15}$, then $F$ is quasihomogenous $\mu-$constant family of type $(3,2,1;15).$
Then as we saw,  the canonical stratification is $(a)$ regular.

$\left\{\begin{array}{lll} \frac{\partial F}{\partial x}=y^7+15x^{14} \\
 \frac{\partial F}{\partial y}=6ty^5+7y^{6} \\  \frac{\partial F}{\partial z}=5z^4+ty^6 \end{array}\right.$

We shall construct an explicit analytic path 
$\gamma(s)=(x(s),y(s),z(s),t(s))$ contained in $V$ along which the $b'$ condition 
fails,   that is   $\Delta(x,y,z,t)=\left 
(\frac{\sum_{i=1}^{n}x_{i}\frac{\partial F}{\partial 
x_{i}}(x,y,z,t)}{\|x\|\| grad_{x} F(x,y,z,t)\|}\right)$ do not tends
to $0$ when $(x,y,z,t)$  tends to $0$ along $\gamma(s).$ 

Let's take it of the following form
$$\left\{\begin{array}{lll} x(s)=\lambda s^5 \\
y(s)=\alpha s^5\\z(s)=s^8\\t(s)=-\frac{5}{\alpha^6}s^2 \end{array}\right.$$
with  $\alpha\not =0.$\\
We must have $ F(\gamma(s))=(1-\frac{5}{\alpha^6}\alpha^6+\lambda \alpha^7+
 \lambda^{15}s^{35})s^{40}\equiv 0 $ i.e.\\
$G(\lambda,s)=-4+\lambda \alpha^7+ \lambda^{15}s^{35}$

Since $\frac{\partial G}{\partial \lambda}(\lambda,0)=\alpha^7\not=0$,  by the implicit function theorem
$\lambda $ is a smooth function of $s.$

Then we have along $\gamma(s)$  nears $s=0:$

$\left\{\begin{array}{lll} \frac{\partial F}{\partial x}=y^7+15x^{14}=\alpha^7s^{35} + 15\lambda^{15}s^{70}
\sim  \alpha^7s^{35}\\\\

\frac{\partial F}{\partial y}=6ty^5+7y^{6} 
=\left(\frac{-30}{\alpha}+7\alpha^6\lambda\right)s^{35}\\\\
 
 \frac{\partial F}{\partial z}=5z^4+ty^6=5s^{32}- \frac{5}{\alpha^6}\alpha^6s^{32}\equiv 0\end{array}\right.$

The limit of orthogonal secant vectors $\frac{(x,y,z)}{\|(x,y,z)\|}$ is  $(1:\alpha:0)$
and the limit of normal vectors $\left ( \frac{ grad_{x} F(x,y,z,t)}{\| grad_{x} F(x,y,z,t)\|}\right)$ is
$(\alpha^7:\frac{-30}{\alpha}+ 7\alpha^6:0).$

Then $\Delta(\gamma(s))$ tends to $0$ if and only if 
$8\alpha^7-30=0$.  We can clearly choose $\alpha\not=0$ such that 
$8\alpha^7-30\not=0$, this means that Whitney
condition fails along this curve.

Now the construction of the family analytique automorphisms in the 
proof gives \`a ``control function'' $\rho$ such that the family  
$\{f_{t}(x,y,z)=z^5+ty^6z+y^7x+x^{15}\}$, has a $\rho$-vanishing fold 
i.e. its $\rho$-Milnor radius is not uniform.

The control function is obtained by composing the standard disytnace 
function $|x|^2+|y|^2+|z|^2$ by the family of automorphisms obtained 
this way.

It may be interesting to establish an analogue of theorem 
\ref{mainthm} for the condition $C$ 
 (see \cite{B}) and the relation with  uniform radii.

\section{tame mappings}
In this section we establish a version of Theorem \ref{mainthm} for 
family of germs of tame mappings (i.e. in the $o$-minimal category ).
We will assume the reader familiar with the basic facts about $o$-minimal structure. The standard references are L. Van den Driess \cite{D}, L. Van den Driess and C. Miller \cite{DM}
and M. Coste \cite{C}.

Let us first recall the definition of an $o$-minimal structure extending the field $(\R,+,.)$.

\begin{defn} Let $\cS=\cup_{n\in\N}\cS_{n}$, where for each $n\in \N$, $\cS_{n}$ is a family of subsets of $\R^{n}$.

We say that the collection $\cS$ is an \textit{$o$-minimal structure} on  $(\R,+,.)$ if:
\begin{itemize}
\item[1)] each $\cS_{n}$ is a boolean algebra
\item[2)] if $A\in\cS_{n}$ and $B\in\cS_{m}$, then $A\times B\in \cS_{n+m}$
\item[3)] let $A\in\cS_{n+m}$ and $\pi: \R^{n+m}\ra\R^n$ be the projection on the first $n$ coordinates, then $\pi(A)\in\cS_{n}$.
\item[4)] all algebraic subsets of $\R^{n}$ are in $\cS_{n}$
\item[5)] the elements of $\cS_{1}$ are the finite unions of points and
intervals.
\end{itemize}

A subset $A$ of $\R^{n}$ which belongs to  $\cS_{n}$ is called
a \textit{definable} set in $\mathcal S$.
A map $f:A\ra\R^m$ is \textit{definable} in $\mathcal S$
if its graph is a definable subset of $\R^n \times\R^m$ in $\mathcal S$,
if in addition, it is $C^{k}$ for some $k\in \N$, we call it
a $C^{k}$ definable map in $\mathcal S$.

We call a tame map map definable in some $o$-minimal structure.
\end{defn}

Let $\cS$ be an $o$-minimal structure on $(\R,+,.)$.
We recall  from \cite{DM} the following notation:

\vspace{2mm}

\noindent {\bf Notation:}
Let $p$ be a natural number.
Let $\Phi^p_{\cS}$ denote the set of all odd,
strictly increasing bijections $\phi:\R\ra\R$ $C^p$ definable in $\cS$
and $p$-flat at $0$ (that is $\phi^{(l)}(0)=0$ for $l=0,\ldots, p$).
\vspace{2mm}

\noindent
We quote also from this paper the following lemma (Lemma C.7. page 523):
\begin{lem}\label{Lem0}
Let $f:A\times \R^*\raa \R$ be a definable function in $\cS$,
$A\subset \R^{n}$.
Then, for any $p\in\mathbb{N}$,
there exists $\phi\in \Phi^p_{\cS}$ such that
$\lim_{t\to 0}\phi(t)f(x,t)=0$ for each $x\in A$.
\end{lem}

A $C^k$ version of this lemma is given in the following:

\begin{lem}\label{Lem1}
Let $f:U\times \R^*\raa \R$ be a  $C^k$ definable
function in $\cS$, and let $U$ be an open subset of $ \R^{n}$.
Then, for any $p\in\mathbb{N}$,
there exists $\phi\in \Phi^p_{\cS}$ such that the function
$$
g(x,t)=\begin{cases}
\phi(t)f(x,t)&\, \text{if $t\not=0$},\\
0&\, \text{if $t=0$}.
\end{cases}
$$
is a $C^k$ definable function.
\end{lem}
\begin{proof}
Let $h:U\times \R^*\ra\R$ denote any function from the collection of partial
derivatives:
$$\{D_{(x,t)}^{\alpha}f\ ;\ \,
\alpha=(\alpha_{1},\ldots,\alpha_{n},\alpha_{n+1})\in \N^{n+1}\,\text{and}\, \vert \alpha\vert\le k\}$$

By Lemma \ref{Lem0}, there exists  $\theta\in \Phi^p_{\cS}$ such that  $\lim_{t\to 0}\theta(t)h(x,t)=0$ for each $x\in U$.
Then $\phi:= \theta^{2k+1}$ satisfies the needs.
\end{proof}

\begin{defn} A structure $\cS$ on the field $(\R,+,.)$ is \textit{polynomially
bounded} if for any function $f:\R\raa\R$ definable in $\cS$, there exists
$N\in \N$ such that
$$\vert f(t)\vert\le t^N$$
for all sufficiently large $t$.
\end{defn}
\begin{lem}\label{Lem2}
If $\cS$ is polynomially bounded, then for any $\phi\in \Phi^p_{\cS}$,
there exists $d\in \N$ and real number $\epsilon>0$ such that:
$$\vert \phi(t)\vert \geq t^d$$
for any $t\in (-\epsilon, \epsilon)$.
\end{lem}
\begin{proof}
We take $\displaystyle\theta(s):=\frac{1}{\phi(\frac{1}{s})}$
where $s=\frac{1}{t}$.
Since $\theta$ is definable in a \textit{polynomially bounded} structure,
there exists $d\in \N$ and $M\in \N$ such that $\vert \theta(s)\vert\leq s^d$
for $\vert s\vert > M$.
Therefore $\vert \phi(t)\vert\geq t^d$ for any $t\in (-\epsilon, \epsilon)$ with $\epsilon=\frac{1}{M}$.
\end{proof}
\begin{defn}
 A tame map is a map definable in some polynomially bounded $o$-minimal structure.
\end{defn}

Using the lemmas above, we can show the following theorem which is a 
tame version of the main result.

\begin{thm}\label{Th3}

    Given a  tame family of isolated singularities germs \\
$F:(\mathbb{R}^p\times \mathbb{R}^n, 
\mathbb{R}^p\times\{0\})\to (\mathbb{R}^k,0).$ We suppose the family 
$F$ is a $\mu$-constant deformation of $f_{0}=f$. Then the  following conditions are equivalent
\begin{itemize}
\item[(i)] $F$ is whitney regular.
\item[(ii)] $\forall \Phi\in \mathcal{G}$, $F\circ \Phi$ has no vanishing fold.
\end{itemize}
\end{thm}

\end{document}